\documentclass[12pt, a4paper]{article}
\usepackage{graphicx}
\usepackage{amssymb}
\usepackage{stmaryrd}
\usepackage{enumitem}

\usepackage{geometry}
\usepackage{tikz}
\usepackage{float}
\usepackage{url}
\usepackage{mathtools}
\usepackage{amsmath,amsfonts,amsthm}
\usepackage{mathrsfs} 
\newcommand{\irr}{\operatorname{irr}}
\newcommand{\case}[1]{\paragraph*{Case #1:}}
\newtheorem{hypothesize}{Hypothesize}

\newtheorem{theorem}{Theorem}[section]
\newtheorem{lemma}[theorem]{Lemma}
\newtheorem{proposition}{Proposition}[section]

\newtheorem{definition}{Definition}
\newtheorem{example}{Example}

\usepackage{subfig}

\DeclarePairedDelimiter\ceil{\lceil}{\rceil}
\DeclarePairedDelimiter\floor{\lfloor}{\rfloor}

\usepackage{ytableau}
\usepackage[colorlinks,
linkcolor=blue,
anchorcolor=blue,
citecolor=blue
]{hyperref}
\begin{document}
\begin{center}
{\large \bf Extremal Irregularity Bounds for Albertson Index in Trees and Bipartite Graphs with Degree Sequence}
\end{center}
\begin{center}
 Jasem Hamoud$^1$ \hspace{0.2 cm} Alexey Belov Yakovlevich$^{2}$ \hspace{0.2 cm}  Duaa Abdullah$^{3}$\\[6pt]
 $^{1,2,3}$ Physics and Technology School of Applied Mathematics and Informatics \\
Moscow Institute of Physics and Technology, 141701, Moscow region, Russia\\[6pt]
Email: 	 $^{1}${\tt jasem1994hamoud@gmail.com}, $^{2}${\tt kanelster@gmail.com}, 
 $^{3}${\tt duaa1992abdullah@gmail.com}
\end{center}
\noindent
\begin{abstract}
In this paper, we investigate extremal bounds on the Albertson index, a topological invariant that quantifies degree irregularity in graphs. The focus lies primarily on connected graphs and trees. We establish necessary and sufficient conditions for a connected graph $G$ of order $n \geq 4$ to achieve the maximum Albertson index, specifically that 
$$
2(\delta - 1)(2\delta - 1) > 0,
$$
where $\delta$ denotes the minimum vertex degree. For trees $T$ with degree sequence $\mathscr{D} = (d_1, \ldots, d_n)$ and maximum degree $\Delta$, we derive explicit upper and lower bounds on the minimum and maximum Albertson indices. These bounds utilize combinatorial and factorial expressions involving vertex degrees and tree parameters. Furthermore, the relationship between the Albertson index $\irr(T)$ of a tree and its first Zagreb index $M_1(T)$ and degree sequence is given by
$$
\irr(T) = M_1(T) + \sum_{i=2}^{n-1} d_i + d_n - d_1 - 2n - 2,
$$
which offers a structural interpretation of irregularity in terms of vertex degrees. Additionally, we present inequalities linking the minimum and maximum Albertson indices through the maximum degree $\Delta$, refining the understanding of how degree sequences influence irregularity measures.

\end{abstract}

\noindent\textbf{AMS Classification 2010:} 05C05, 05C12, 05C20, 05C25, 05C35, 05C76, 68R10.


\noindent\textbf{UDC:} 519.172.1

\section{Introduction}\label{sec1}

Throughout this paper. Let $G=(V,E)$ be a simple, connected graph, where $n=|V(G)|$, $m=|E(G)|$. The minimum and the maximum degree of $G$ are denoted by $\delta, \Delta$ where $1\leqslant \delta <\Delta \leq n-1$.  Let $G = (V_1, V_2, E)$ be a graph with two vertex sets $V_1 = \{v_1, \dots, v_i\}$ where $n_1=|V_1|$ and $V_2 = \{v_1, \dots, v_j\}$ where $n_2=|V_2|$. Then, if 
\[
\sum_{i=1}^{n_1}  d_G(v_i) = \sum_{j=1}^{n_2}  d_G(v_j),
\]
we say $G$ is bipartite graph~\cite{KunegisJ}. Optimally, $G$ is bipartite graph if its vertices can be partitioned into two sets such that every edge connects a vertex in $V_1(G)$ to a vertex in $V_2(G)$ (see~\cite{DulmageMendelsohnj, Hanary, Thulasiraman, ArmenAsratian,Reinhard6thDiestel}).

The maximum degree of a node had been defined in~\cite{PavlopoulosKontou} where the symmetric nodes are equal and we denote that by $(\max[ d(u)] = |V| \text{ or } \max[ d(V)] = |u|)$. Furthermore, we say: 
\[
\sum[ d(u)] = |V| \text{ or } \sum[ d(V)] = |u|.
\]

Let $\mathcal{P}_n$ be a path of order $n$ and $\mathcal{S}_n$ the star tree of order $n$, typically denoted by $K_{1,n-1}$.  In 1997, Albertson in~\cite{Albertson1997} mention to the imbalance of an edge $uv$ by $imb(uv)$. Then, the irregularity (Albertson index) measure defined in~\cite{Albertson1997, GutmanHansenMelot2005, AbdoBrandtDimitrov,HamidIS,Lina2024Zhoub} as: 
\begin{equation}~\label{eqsec1n2}
    \operatorname{irr}(G)=\sum_{uv\in E(G)}\lvert  d_G(u)- d_G(v) \rvert.
\end{equation}

Denote by $\irr_{\max}$ the maximum degree of Albertson index and $\irr_{\min}$ the minimum degree of Albertson index. Beyond generic graphs, special attention has been devoted to irregularity in trees with constrained degree sequences. Furthermore see~\cite{Mandal2022Prvanovic,broutin2012asymptotics,Michael2007Henning,NasiriFathTabarGutman2013,RuixiaWangJiaWangQiaopingGuoWeiMeng}. Molloy and Reed~\cite{molloy1995critical} and later Broutin and Marckert~\cite{broutin2012asymptotics} studied the asymptotic behavior of trees with fixed degree sequences. Zhang et al.~\cite{zhang2012number} characterized tree families via degree partitions and enumerated the number of realizable trees under such constraints. These results are instrumental in analyzing extremal indices among structured tree families.
Ali, A. et al in~\cite{AliDimitrovEtAl2025} mention to total irregularity of $\mathscr{D}(G)=( d(v_1), d(v_2),\dots, d(v_i))$ a degree sequence of $G$ where $ d(v_1)\geqslant  d(v_2)\geqslant \dots\geqslant d(v_i)$ as 
\begin{equation}~\label{eqsec1n3}
    \operatorname{irr}(G)=2(n+1)m - 2 \sum_{i=1}^n i d(v_i),
\end{equation}
when $\mathscr{D}(G)=(k, k, \ldots, k)$. Then $G$ is regular of degree $k$. Otherwise, the graph is irregular. Ghalavand, A. et al. In~\cite{GhalavandAshrafiR2023} proved  $\operatorname{irr}_T(G) \leq n^2\operatorname{irr}(G)/4$, when the bound is sharp for infinitely many graphs.  The first and the second Zagreb index, $M_1(G)$ and $M_2(G)$ are defined in~\cite{GutmanTrinajstic1972, GutmanTrinajsticWilcox1975} as: 
\[
M_1(G)=\sum_{v\in V(G)} d(v)^2, \quad \text{and} \quad M_2(G)=\sum_{uv\in E(G)}  d(u)  d(v).
\]
Sigma index had defined in~\cite{GutmanToganYurttas2016,AbdoDimitrovGutman} as: 
\[
\sigma(G)=\sum_{uv\in E(G)} ( d_G(u)- d_G(v))^2.
\]
Denote by $\sigma_{\max}$ the maximum degree of Sigma index and $\sigma_{\min}$ the minimum degree of Sigma index. Bipartite graphs are additionally employed to analyse algebraic graphs, which have associative ring elements as vertices and edges. The study of algebraic graphs has applications in extreme graph theory and algebraic cryptography. Additionally, bipartite graphs are utilized in many fields of mathematics (see~\cite{KunegisJ,BrauchTimothy,HamidIS,AsratianDenley}).

Our goal of this paper, study of some bounds on trees among Albertson index by considering the special trees known caterpillar tree. Also, study of the extreme bounds of the two indices — Albertson index and Sigma index — as well as the study of these indices on bipartite graph theory, and discussion of these cases when the bipartite graph is complete and when it is incomplete.

This paper is organized as follows. In Section~\ref{sec1}, we review the important concepts relevant to our work, including a literature survey of the most related papers. In Section~\ref{sec2}, we provide a preface covering some of the key theories and properties we have utilized to understand the work. In Section~\ref{sec3}, we introduce some bounds involving the extremal values of the Albertson and Sigma Indices. In Section~\ref{sec4}, we introduce some bounds involving the extremal values of the Albertson and Sigma Indices in bipartite graphs.

 \section{Preliminaries}\label{sec2}
Through this paper, we point out that all the trees included in the study are caterpillar trees, we consider caterpillar trees denoted by $\mathscr{C}(n, m)$, where $n$ is the number of backbone (or path) vertices and $m$ is the number of pendant vertices attached to each.  
We denote by $\mathscr{C}(n,m)$ a caterpillar tree with a pair $(n,m)$ of vertices, where we refer to the $n$ main vertices and $m$ to pendent vertices. In fact, that is shown in Figure~\ref{figcater} where we consider the main vertices given by sequence $X = (x_1, x_2, \dots, x_n)$ and the pendent vertices given with value $m=3$ for each vertices. 

\begin{figure}[H]
\centering
\begin{tikzpicture}[scale=.9]
\draw  (2,2)-- (4,2);
\draw  (4,2)-- (6,2);
\draw [line width=2pt,dash pattern=on 1pt off 1pt] (6,2)-- (8,2);
\draw  (2,2)-- (1.35,1);
\draw  (2,2)-- (2,1);
\draw  (2,2)-- (2.77,1.02);
\draw  (4,2)-- (3.59,1);
\draw  (4,2)-- (4,1);
\draw  (4,2)-- (4.55,0.98);
\draw  (6,2)-- (5.59,1.06);
\draw  (6,2)-- (6,1);
\draw  (6,2)-- (6.53,1.04);
\draw (8,2)-- (7.55,1.02);
\draw  (8,2)-- (8,1);
\draw  (8,2)-- (8.59,1.02);
\draw  (3,2)-- (2,3);
\draw  (3,2)-- (3,3);
\draw  (3,2)-- (4,3);
\draw (1.6,2.5) node[anchor=north west] {$x_1$};
\draw (2.8811628144708608,2.006163333401493) node[anchor=north west] {$x_2$};
\draw (3.8699661128105616,2.538595878661331) node[anchor=north west] {$x_3$};
\draw (4.869635381461688,2.0170293037129183) node[anchor=north west] {$x_4$};
\draw (7.879509157726492,2.484266027104205) node[anchor=north west] {$x_n$};
\draw  (5,2)-- (4.325669695313298,3.0458367583511627);
\draw  (5,2)-- (5,3);
\draw (5,2)-- (5.66201453448175,3.029927891218205);
\draw (5.836706739178537,2.5711937895956067) node[anchor=north west] {$x_5$};
\begin{scriptsize}
\draw [fill=black] (2,2) circle (1pt);
\draw [fill=black] (4,2) circle (1pt);
\draw [fill=black] (6,2) circle (1pt);
\draw [fill=black] (8,2) circle (1pt);
\draw [fill=black] (1.35,1) circle (1pt);
\draw [fill=black] (2,1) circle (1pt);
\draw [fill=black] (2.77,1.02) circle (1pt);
\draw [fill=black] (3.59,1) circle (1pt);
\draw [fill=black] (4,1) circle (1pt);
\draw [fill=black] (4.55,0.98) circle (1pt);
\draw [fill=black] (5.59,1.06) circle (1pt);
\draw [fill=black] (6,1) circle (1pt);
\draw [fill=black] (6.53,1.04) circle (1pt);
\draw [fill=black] (7.55,1.02) circle (1pt);
\draw [fill=black] (8,1) circle (1pt);
\draw [fill=black] (8.59,1.02) circle (1pt);
\draw [fill=black] (3,2) circle (1pt);
\draw [fill=black] (2,3) circle (1pt);
\draw [fill=black] (3,3) circle (1pt);
\draw [fill=black] (4,3) circle (1pt);
\draw [fill=black] (5,2) circle (1pt);
\draw [fill=black] (4.325669695313298,3.0458367583511627) circle (1pt);
\draw [fill=black] (5,3) circle (1pt);
\draw [fill=black] (5.66201453448175,3.029927891218205) circle (1pt);
\end{scriptsize}
\end{tikzpicture}
\caption{Example of caterpillar tree $\mathscr{C}(n,3)$ in graph theory.}\label{figcater}
 \end{figure}

\begin{definition}[\cite{WagnerWang}]~\label{WagnerWang}
Chemical graph theory is a multidisciplinary field in which the molecular framework of a chemical molecule is represented as a graph, and related mathematical questions are studied using graph-theoretical and computational approaches.
\end{definition}

Zhang, P. and  Wang, X in\cite{j1} show ``caterpillar trees'' are employed in chemical graph theory for describing molecular structures as we show in Definition~\ref{WagnerWang}, and several topological indices, which include the Albertson index, have been used to study the characteristics of molecules. In Figure~\ref{fig2} we show some caterpillars trees $\mathscr{C}(n,3)$ for $n=3,4,5,6$: 
\begin{figure}[H]
\centering
\begin{tikzpicture}[scale=.7]
\draw  (1,5)-- (3,5);
\draw  (3,5)-- (5,5);
\draw  (1,5)-- (0.34,4.01);
\draw  (1,5)-- (1,4);
\draw  (1,5)-- (1.46,4.01);
\draw  (3,5)-- (2.5,4.07);
\draw  (3,5)-- (3,4);
\draw  (3,5)-- (3.42,4.01);
\draw  (5,5)-- (4.38,4.01);
\draw  (5,5)-- (5,4);
\draw  (5,5)-- (5.56,4.01);
\draw  (7,5)-- (9,5);
\draw  (9,5)-- (11,5);
\draw  (11,5)-- (13,5);
\draw  (7,5)-- (6.58,4.05);
\draw  (7,5)-- (7,4);
\draw  (7,5)-- (7.58,3.99);
\draw  (9,5)-- (8.52,4.01);
\draw  (9,5)-- (9,4);
\draw  (9,5)-- (9.54,4.03);
\draw  (11,5)-- (10.6,4.03);
\draw  (11,5)-- (11,4);
\draw  (11,5)-- (11.36,4.01);
\draw  (13,5)-- (12.56,4.01);
\draw  (13,5)-- (13,4);
\draw  (13,5)-- (13.58,4.03);
\draw  (1,2)-- (2,2);
\draw  (2,2)-- (3,2);
\draw  (3,2)-- (4,2);
\draw  (4,2)-- (5,2);
\draw  (1,2)-- (0.62,1.09);
\draw  (1,2)-- (1,1);
\draw  (1,2)-- (1.48,1.05);
\draw  (2,2)-- (1.24,2.95);
\draw  (2,2)-- (2,3);
\draw  (2,2)-- (2.82,2.97);
\draw  (3,2)-- (2.58,1.05);
\draw  (3,2)-- (3,1);
\draw  (3,2)-- (3.62,1.01);
\draw  (4,2)-- (3.38,2.97);
\draw  (4,2)-- (4,3);
\draw (4,2)-- (4.64,3.01);
\draw  (5,2)-- (4.66,1.07);
\draw  (5,2)-- (5,1);
\draw  (5,2)-- (5.52,1.01);
\draw  (7,2)-- (8,2);
\draw  (8,2)-- (9,2);
\draw  (9,2)-- (10,2);
\draw  (10,2)-- (11,2);
\draw  (11,2)-- (12,2);
\draw  (7,2)-- (6.5,1);
\draw  (7,2)-- (7,1);
\draw  (7,2)-- (7.5,1);
\draw  (8,2)-- (7.5,3);
\draw  (8,2)-- (8,3);
\draw  (8,2)-- (8.5,3);
\draw  (9,2)-- (8.5,1);
\draw  (9,2)-- (9,1);
\draw  (9,2)-- (9.5,1);
\draw (10,2)-- (9.48,2.97);
\draw  (10,2)-- (10,3);
\draw  (10,2)-- (10.54,2.99);
\draw  (11,2)-- (10.5,1);
\draw  (11,2)-- (11,1);
\draw  (11,2)-- (11.52,1.03);
\draw  (12,2)-- (11.5,3);
\draw (12,2)-- (12,3);
\draw  (12,2)-- (12.5,3);
\begin{scriptsize}
\draw [fill=black] (1,5) circle (2pt);
\draw [fill=black] (3,5) circle (2pt);
\draw [fill=black] (5,5) circle (2pt);
\draw [fill=black] (0.34,4.01) circle (2pt);
\draw [fill=black] (1,4) circle (2pt);
\draw [fill=black] (1.46,4.01) circle (2pt);
\draw [fill=black] (2.5,4.07) circle (2pt);
\draw [fill=black] (3,4) circle (2pt);
\draw [fill=black] (3.42,4.01) circle (2pt);
\draw [fill=black] (4.38,4.01) circle (2pt);
\draw [fill=black] (5,4) circle (2pt);
\draw [fill=black] (5.56,4.01) circle (2pt);
\draw [fill=black] (7,5) circle (2pt);
\draw [fill=black] (9,5) circle (2pt);
\draw [fill=black] (11,5) circle (2pt);
\draw [fill=black] (13,5) circle (2pt);
\draw [fill=black] (6.58,4.05) circle (2pt);
\draw [fill=black] (7,4) circle (2pt);
\draw [fill=black] (7.58,3.99) circle (2pt);
\draw [fill=black] (8.52,4.01) circle (2pt);
\draw [fill=black] (9,4) circle (2pt);
\draw [fill=black] (9.54,4.03) circle (2pt);
\draw [fill=black] (10.6,4.03) circle (2pt);
\draw [fill=black] (11,4) circle (2pt);
\draw [fill=black] (11.36,4.01) circle (2pt);
\draw [fill=black] (12.56,4.01) circle (2pt);
\draw [fill=black] (13,4) circle (2pt);
\draw [fill=black] (13.58,4.03) circle (2pt);
\draw [fill=black] (1,2) circle (2pt);
\draw [fill=black] (2,2) circle (2pt);
\draw [fill=black] (3,2) circle (2pt);
\draw [fill=black] (4,2) circle (2pt);
\draw [fill=black] (5,2) circle (2pt);
\draw [fill=black] (0.62,1.09) circle (2pt);
\draw [fill=black] (1,1) circle (2pt);
\draw [fill=black] (1.48,1.05) circle (2pt);
\draw [fill=black] (1.24,2.95) circle (2pt);
\draw [fill=black] (2,3) circle (2pt);
\draw [fill=black] (2.82,2.97) circle (2pt);
\draw [fill=black] (2.58,1.05) circle (2pt);
\draw [fill=black] (3,1) circle (2pt);
\draw [fill=black] (3.62,1.01) circle (2pt);
\draw [fill=black] (3.38,2.97) circle (2pt);
\draw [fill=black] (4,3) circle (2pt);
\draw [fill=black] (4.64,3.01) circle (2pt);
\draw [fill=black] (4.66,1.07) circle (2pt);
\draw [fill=black] (5,1) circle (2pt);
\draw [fill=black] (5.52,1.01) circle (2pt);
\draw [fill=black] (7,2) circle (2pt);
\draw [fill=black] (8,2) circle (2pt);
\draw [fill=black] (9,2) circle (2pt);
\draw [fill=black] (10,2) circle (2pt);
\draw [fill=black] (11,2) circle (2pt);
\draw [fill=black] (12,2) circle (2pt);
\draw [fill=black] (6.5,1) circle (2pt);
\draw [fill=black] (7,1) circle (2pt);
\draw [fill=black] (7.5,1) circle (2pt);
\draw [fill=black] (7.5,3) circle (2pt);
\draw [fill=black] (8,3) circle (2pt);
\draw [fill=black] (8.5,3) circle (2pt);
\draw [fill=black] (8.5,1) circle (2pt);
\draw [fill=black] (9,1) circle (2pt);
\draw [fill=black] (9.5,1) circle (2pt);
\draw [fill=black] (9.48,2.97) circle (2pt);
\draw [fill=black] (10,3) circle (2pt);
\draw [fill=black] (10.54,2.99) circle (2pt);
\draw [fill=black] (10.5,1) circle (2pt);
\draw [fill=black] (11,1) circle (2pt);
\draw [fill=black] (11.52,1.03) circle (2pt);
\draw [fill=black] (11.5,3) circle (2pt);
\draw [fill=black] (12,3) circle (2pt);
\draw [fill=black] (12.5,3) circle (2pt);
\end{scriptsize}
\end{tikzpicture}
\caption{Caterpillars tree with $(n,3)$ vertices.}\label{fig2}
\end{figure} 

\begin{proposition}[\cite{dorjsembe2022irregularity}]
Let $ G $ be a graph with minimum degree $ \delta $, maximum degree $ \Delta $, and vertex set of size $ n $. Then:
\[
\operatorname{irr}(G) > \frac{ \delta(\Delta - \delta)^2 \cdot n }{ \Delta + 1 }.
\]
\end{proposition}
The next formula captures the Albertson index in caterpillar trees based on the degrees of the backbone vertices, and will be instrumental in deriving our closed-form expressions.

\begin{proposition}[\cite{HamoudwithDuaa,HamoudwithDuaaPn2}]~\label{Caterpillar}
For Caterpillar tree with path vertices with degrees $d_1,d_2,\dots , d_n$, we have:
	\[\irr(G)=\left( {{d_n} - 1} \right)^2 + \left( {d_1 - 1} \right)^2 + \sum\limits_{i = 2}^{n - 1} {\left( {{d_i} - 1} \right)\left( {{d_i} - 2} \right)} +\sum_{i=1}^{n-1}|d_i-d_{i+1}|.\]
\end{proposition}
\begin{definition}[$\sigma$-Maximum and Minimum]~\label{maxsigmadef}
    Let $\mathcal{S}$ be a class of graphs, then maximum Sigma $\sigma_{\max}$, minmum Sigma $\sigma_{\min}$ of a graph $G$, given by: \begin{gather*} \sigma_{\max}(\mathcal{S}) =\max \{\sigma(G)\mid G\in \mathcal{S}\}, \\
\sigma_{\min}(\mathcal{S}) =\min \{\sigma(G)\mid G\in \mathcal{S}\}.
\end{gather*}
\end{definition}

\section{Bounds on Albertson Index with Given Degree Sequence}\label{sec3}
In this section, the study of the extreme values (maximum and minimum) for each index requires certain conditions~\cite{YangDeng2023}. Among the essential conditions commonly considered are the maximum degree $\Delta$ and the minimum degree $\delta$. Through Proposition~\ref{proboundmindel} we established the maximum value with the minimum degree $\delta$ with condition $2(\delta-1)(2\delta-1)>0$. 

\begin{proposition}~\label{proboundmindel}
Let $G$ be a connected graph of order $n\geq 4$. Then, $G$ has a maximum value of Albertson index if and only if $2(\delta-1)(2\delta-1)>0$. 
\end{proposition}
\begin{proof}
Assume $G$ has the maximum value of Albertson index among all connected graph of order $n\geq 4$ with the minimum degree $\delta\geqslant 1$. Then, $2(\delta-1)>0$ and $(2\delta-1)>0$. Thus, we have $2(\delta-1)(2\delta-1)>0$. Thus, no two vertices of $\delta$ in $G$ are adjacent.
\end{proof}
\begin{proposition}~\label{boundstep2Albertsonn2}
Let $T$ be a tree of order $n$, let $\alpha$ be an integer where $1\leq \alpha\leq \Delta-3$. Let $\mathscr{D}=(d_1,\dots,d_n)$ be an increasing degree sequence.  Then, the upper bound of the minimum Albertson index is:
\begin{equation}~\label{eq1boundstep2Albertsonn2}
 \irr_{\min}\geqslant   2^{\alpha}\dfrac{2^{n-\Delta}+\ceil{\frac{n\Delta^2}{2}}}{(d_n-d_1)!+(\Delta-1)!}.
\end{equation}
\end{proposition}
\begin{proof}
Let $\mathscr{D} = (d_1, \dots, d_n)$ be a degree sequence where $d_1 \leqslant \dots \leqslant d_n$, with the condition $d_n - d_1 \geqslant \Delta - 1$. Let $p < n$ be an integer satisfying $2^p \leq 2^{n - \Delta}$. When $\Delta < n$, it holds that $1 \leq \alpha \leq p \leq \Delta - 3$. Then, the bound~\eqref{eq2boundstep2Albertsonn2} provides a term for the minimum Albertson index satisfying

\begin{equation} \label{eq2boundstep2Albertsonn2}
  0 < \frac{2^{n-\Delta}}{(\Delta - 1)!} \leq 1.
\end{equation}

For the term $\lceil n \Delta^2 / 2 \rceil$, we note that $\lceil n \Delta^2 / 2 \rceil < 2^p$. Therefore, the factorial terms satisfy $(d_n - d_1)! \geqslant (d_{n-1} - d_2)! \geqslant \dots \geqslant (d_{i+1} - d_i)!$ for $1 \leq i \leq n$. Using the inequality $\Delta (\Delta - 1)^p \leq 2^{n+p}$, there are at most $\Delta (\Delta - 1)^p$ vertices, where

\[
\Delta (\Delta - 1)^p \leq n^p \leq 1 + \Delta + \Delta^2 + \dots + \Delta^n + \Delta (\Delta - 1)^p.
\]

The lower and upper bounds of the minimum Albertson index are then:

\begin{equation} \label{eq3boundstep2Albertsonn2}
  \frac{2^{n-\Delta}}{(\Delta - 1)!} \leq \irr_{\min} \leq \frac{2^{n-\Delta}}{(d_{i+1} - d_i)!}.
\end{equation}

From~\eqref{eq2boundstep2Albertsonn2} and \eqref{eq3boundstep2Albertsonn2}, the term in \eqref{eq1boundstep2Albertsonn2} satisfies

\[
\log_2 \left( 2^{n - \Delta} + \lceil n \Delta^2 / 2 \rceil \right) = (n - \Delta) + \log_2 \left( 1 + \lceil n \Delta^2 / 2 \rceil \cdot 2^{-(n - \Delta)} \right).
\]

Considering $1 \leq \alpha \leq p \leq \Delta - 3$ and $\irr_{\min} > \Delta$, the bound

\begin{equation} \label{eq4boundstep2Albertsonn2}
  3 \leq \frac{2^{n-\Delta} + \lceil \frac{n \Delta^2}{2} \rceil}{(d_n - d_1)! + (\Delta - 1)!} \leq \Delta - 3
\end{equation}

holds. Hence, from~\eqref{eq4boundstep2Albertsonn2}, the upper bound for the Albertson index can be compared with $\Delta$. When $1 \leq \alpha \leq \Delta - 3$, then $\alpha \geq 3$. This establishes that the term $2^\alpha$ grows the upper bound of the minimum Albertson index for $\alpha < p < n - \Delta$, considering that $2^\alpha < 2^p \leq 2^{n - \Delta}$. Thus,

\begin{equation} \label{eq5boundstep2Albertsonn2}
  \irr_{\min} \geq \frac{2^\alpha}{\Delta - 3}.
\end{equation}

As desired.
\end{proof}

\begin{proposition}~\label{boundalbertsonn02}
Let $T$ be a tree of order $n$, let $\mathscr{D}=(d_1,\dots,d_n)$ be an increasing degree sequence, the maximum Albertson index $\irr_{\max}$, the minimum Albertson index $\irr_{\min}$. Then, 
\begin{itemize}
    \item The upper bound of the minimum Albertson index is: 
    \begin{equation}~\label{eq1boundalbertsonn02}
\irr_{\min}\geqslant \frac{(\Delta-2)^3}{n\Delta-\delta}.
\end{equation}
\item The lower bound of the minimum Albertson index is: 
    \begin{equation}~\label{eq2boundalbertsonn02}
\irr_{\min}\leqslant \frac{\delta}{\Delta+1}n\Delta^2+\frac{\Delta^2(\Delta-\delta)}{6\delta(\Delta-1)}.
\end{equation}
\end{itemize}
\end{proposition}

\begin{proposition}~\label{boundalbertsonn1}
Let $T=(V,E)$ be a tree, where $n=|V|$ vertices and $m=|E|$ with maximum degree $\Delta$ and minmum degree $\delta$, then a lower bound of Albertson index given by
\begin{equation}~\label{eqboundalbertsonn01}
\log_{n}(\irr(T)) \leqslant \frac{n(\Delta^2-\delta)+m}{m(\Delta-1)}.
\end{equation}
\end{proposition}
\begin{proof}
Let $T$ be a tree. We know that $m = n - 1$, and that $\Delta(T) > n - m$ while $\delta \leqslant n - m$. We need to prove inequality~\eqref{eqboundalbertsonn01}. To do this, we first determine whether the bounds involving the Albertson index satisfy the lower bound~\eqref{eqboundalbertsonn01}.

We observe that the upper bound for the Albertson index is given by
\begin{equation} \label{eqboundalbertsonn02}
    \irr(T) > \frac{n(\Delta^2 - \delta)}{m(\Delta - 1)}.
\end{equation}

This upper bound~\eqref{eqboundalbertsonn02} remains valid when combining $m$ and considering that $\irr(T) \geq \Delta(\Delta - 1)$. By utilizing the properties of the natural logarithm, and considering the base as the number of vertices $n$, we have
\[
\log_n(\irr(T)) = \frac{\ln \irr(T)}{\ln n}.
\]

From inequality~\eqref{eqboundalbertsonn02}, we then derive a lower bound:
\begin{equation} \label{eqboundalbertsonn03}
    \ln \irr(T) \leq \ln n \cdot \frac{n(\Delta^2 - \delta)}{m(\Delta - 1)}.
\end{equation}

Therefore, combining both the upper bound~\eqref{eqboundalbertsonn02} and the lower bound~\eqref{eqboundalbertsonn03}, we obtain inequality~\eqref{eqboundalbertsonn01} for $\log_n(\irr(T))$ by considering $m$, as desired.
\end{proof}

\begin{lemma}~\label{lemmaboundAlbertsonn01}
Let $T$ be a tree with the maximum of Albertson index $\irr_{\max}$ and the minmum Albertson index $\irr_{\min}$. Let $\mathscr{D}=(d_1,\dots,d_n)$ be an increasing degree sequence, where $n=|V(T)|$ and $m=|E(T)|$. Then, 
\begin{equation}~\label{eq1lemmaboundAlbertsonn01}
\begin{cases}
\irr_{\max} <\Delta(T)(\irr_{\min}-n),\\
\irr_{\min} <\Delta(T)(\irr_{\max}-n).
\end{cases}  
\end{equation}
\end{lemma}
\begin{proof}
Let $T$ be a tree of order at least $3$. The fact that $m < 3(\operatorname{irr}(T) - n)$ is valid for $m$, and $n < 3(\operatorname{irr}(T) - m)$ is valid for $n$. Assume $\mathscr{D} = (d_1, \dots, d_n)$ is an increasing degree sequence. Then 
\[
m \leq \Delta(T)\left(\frac{n}{2}\right),
\]
by restricting the maximum degree $\Delta(T)$, where 
\[
\operatorname{irr}(T) \leq n\Delta^2 - m.
\]
The lower bound of the minimum Albertson index is
\begin{equation} \label{eq2lemmaboundAlbertsonn01}
m\,\operatorname{irr}_{\max} \leq \Delta(T)\left(\frac{n}{2}\right) \operatorname{irr}_{\min} \quad \text{and} \quad m\,\operatorname{irr}_{\min} \leq \Delta(T)\left(\frac{n}{2}\right) \operatorname{irr}_{\max}.
\end{equation}

For an increasing degree sequence $\mathscr{D}$, consecutive terms in a degree sequence from~\eqref{eq2lemmaboundAlbertsonn01} use the constant term 
\[
(\Delta+5)\frac{n}{2} + m - 7.
\]
This term scales the maximum vertex degree by a factor proportional to the graph size. We notice that:
\begin{equation} \label{eq3lemmaboundAlbertsonn01}
\operatorname{irr}(T) \leq \sum_{i=1}^{n} |d_i - d_{i+1}| + (\Delta-5)\frac{n}{2} + \sum_{i=2}^{n-1} |d_i + 2|\,|d_i - 1| + m - 7.
\end{equation}

From the bound~\eqref{eq2lemmaboundAlbertsonn01}, and for the minimum Albertson index according to~\eqref{eq3lemmaboundAlbertsonn01}, we have
\begin{equation} \label{eq4lemmaboundAlbertsonn01}
\operatorname{irr}(T) = \sum_{i=1}^{n} |d_i - d_{i+1}| + \Delta(T)\frac{n}{5} + \sum_{i=2}^{n-1} |d_i + 2|\,|d_i - 1| + m + 2.
\end{equation}

Therefore, 
\[
\operatorname{irr}_{\max} - \operatorname{irr}_{\min} \leq \left\lfloor \frac{n}{2} \right\rfloor,
\]
then 
\[
\operatorname{irr}_{\max} \leq \operatorname{irr}_{\min} \left\lfloor \frac{n}{2} \right\rfloor \quad \text{and} \quad \operatorname{irr}_{\min} \leq \operatorname{irr}_{\max} \left\lfloor \frac{n}{2} \right\rfloor.
\]
Then,
\[
\operatorname{irr}_{\min} \leq \Delta\bigl(\operatorname{irr}_{\max} - \left\lfloor \frac{n}{2} \right\rfloor\bigr) \quad \text{and} \quad \operatorname{irr}_{\max} \leq \Delta\bigl(\operatorname{irr}_{\min} - \left\lfloor \frac{n}{2} \right\rfloor\bigr).
\]
Thus, the bound~\eqref{eq2lemmaboundAlbertsonn01} is valid, as desired.
\end{proof}
\subsection{Albertson Index in Trees and Applications}

Lemma~\ref{lemmaboundAlbertsonn1} considers a pair of trees, $T_1$ and $T_2$, with distinct degree sequences: $T_1$ has a degree sequence dominated by vertices of degree 2, resembling a path graph, while $T_2$ is characterized by a degree sequence consisting of prime numbers. Lemma~\ref{lemmaboundAlbertsonn2} focuses on a single tree with general degree properties, introducing a parameterized lower bound that depends on the maximum and minimum degrees, modulated by a scaling factor $\alpha$. In both lemmas, we provide bounds on the Albertson index.
\begin{lemma}~\label{lemmaboundAlbertsonn1}
Let $T_1,T_2$ be a tow trees, where $T_1$ have $\mathscr{D}=(d_1,\dots, d_i)$ an increasing degree sequence occur of $2$. Tree $T_2$ have $\mathscr{B}=(b_1,\dots, b_i)$ an increasing degree sequence of prime number, consider $n_1=|V(T_1)|,n_2=|V(T_2)|$ and $m_1=|E(T_1)|,m_2=|E(T_2)|$. Then, $n_2=n_1-1$, $m_2=m_1-1$ and $\Delta(T_2)=\Delta(T_1)+1$. Then the bound of Albertson index given by: 
\begin{equation}~\label{eq1lemmaboundAlbertsonn1}
\irr_{\min}\leq  5\cdot \frac{n_1\Delta_2^3+n_2\Delta_1^4+m_1\Delta_2^2}{\Delta_1(\Delta_1+\Delta_2)^2} \leq \irr_{\max}.
\end{equation}
\end{lemma} 
\begin{lemma}~\label{lemmaboundAlbertsonn2}
Let $T$ be a tree with $n=|V(T)|,m=|E(T)$  with the maximum degree $\Delta$ and the minimum degree $\delta\geqslant 1$ and $0\leq \alpha \leq 1$. Then, the lower bound of Albertson index is 
 \[\irr(T) \leq \floor*{\frac{3n^2-10n}{2}}2^{\alpha}+\Delta^2-n\delta.\]
\end{lemma}
It is worth noting that the Albertson index has many applications across trees, and one of the most important applications that we refer to in Theorem~\ref{thmalbertsonn3} is through a degree sequence according to a specific condition.
\begin{theorem}~\label{thmalbertsonn3}
    Let  $T$ be a tree of order $n$, a degree sequence is $\mathscr{D}=(d_1,\dots,d_n)$ where $d_n\geqslant \dots \geqslant d_1$, let $M_1(T)$ be the first Zagreb index of $T$. Then, Albertson index of tree $T$ is: 
\begin{equation}~\label{eq1thmalbertsonn3}
 \irr(T)=M_1(T)+\sum_{i=2}^{n-1} d_i+d_n - d_1-2n-2.
\end{equation}
\end{theorem}
\begin{proof}
Let $T = (V, E)$ be a tree with $n$ vertices and $m$ edges, and let $\mathscr{D} = (d_1, \dots, d_n)$ be a degree sequence such that $d_n \geqslant \dots \geqslant d_1$. For a vertex $v_\ell \in V$ whose degree is at least three, to prove~\eqref{eq1thmalbertsonn3}, we need to establish~\eqref{eqzagrebalbertsonn3}.

We note that the first Zagreb index is defined as
\begin{equation} \label{eqzagrebalbertsonn1}
M_1(T) = \sum_{i=1}^n d_i^2 \leq m \left( \frac{2m}{n-1} + n - 2 \right).
\end{equation}

Our goal is to relate this topological index, the first Zagreb index, to the degree sequence in a specific way. For the degree sequence $\mathscr{D} = (d_1, \dots, d_n)$, we have
\begin{equation} \label{eqzagrebalbertsonn2}
\irr(T) = d_1^2 + d_n^2 + \sum_{i=2}^{n-1} d_i^2 + \sum_{i=2}^{n-1} d_i + d_n - d_1 - 2n - 2.
\end{equation}

From~\eqref{eqzagrebalbertsonn1}, it follows that $d_1^2 + d_n^2 \leq \frac{2m}{n-1}$. Since every tree is a connected graph, we also have $M_1(T) \leq m(m+1)$, and it is known that $\irr(T) = m(m+1)$ when $T$ is a star graph. Thus, from~\eqref{eqzagrebalbertsonn1} and~\eqref{eqzagrebalbertsonn2}, we conclude
\begin{equation} \label{eqzagrebalbertsonn3}
M_1(T) = d_1^2 + d_n^2 + \sum_{i=2}^{n-1} d_i^2,
\end{equation}
as desired.
\end{proof}

We completed discussing among Theorem~\ref{Thmcatern2} with degree sequence $\mathscr{D}$ given by etxtend to degree sequence in Theorem~\ref{thmalbertsonn3} as $\mathscr{D}=(d_{i1}, d_{i2}, \dots, d_{ik}, d_{j1}, d_{j2},\dots, d_{j\ell})$ with consideration the term of $\mathscr{D}$ is $d_n > d_1> d_{n-2} >\dots>d_3 > d_2 > d_{n-1}$. 

\begin{theorem}\label{Thmcatern2}
Let $\mathscr{D}=(d_{i1}, d_{i2}, \dots, d_{ik}, d_{j1}, d_{j2},\dots, d_{j\ell})$ be a degree sequences where $\{k,\ell\}\in \mathbb{N}$, with $n\geq 3$ such that $d_n > d_1> d_{n-2} >\dots>d_3 > d_2 > d_{n-1}$. The caterpillar tree with such order has the  maximum value  of $\irr$ among all caterpillar trees with such degrees sequence of path vertices.
\end{theorem}
\begin{proof}
For any degree sequence of length $n$ given by $d_n > d_1 > \dots > d_2 > d_{n-1},$
such that
\[
\irr(G) = (d_n - 1)^2 + (d_1 - 1)^2 + \sum_{i=2}^{n-2} (d_i - 1)(d_i - 2) + \sum_{i=1}^{n-1} |d_i - d_{i+1}|.
\]

Now, we need to prove that for $n+1$, with the sequence $\mathscr{D}$ where $d_{n+1} \geq d_1 \geq \dots \geq d_2 \geq d_n,$
since $d_1 \geq d_n$, we can simplify the expression by modeling the association as $(d_n - 1)^2 = (d_n - 1)(d_n - 2) + (d_n - 1)$. Furthermore, let $d_1 = d_n + r$ where $r \geq 0$ and $r \in \mathbb{N}$, so that 
\[
(d_n - 1)^2 + (d_1 - 1) = d_n^2 - d_n + r.
\]
Also, for the sum $\sum_{i=2}^{n-2} (d_i - 1)(d_i - 2)$, we have
\[
\sum_{i=2}^{n-2} \left( d_i^2 - 3 d_i + 2 \right).
\]
Then, $(d_i - 1)(d_i - 2) = d_i^2 - 3 d_i + 2,$
where we consider $\sum_{i=2}^{n-2} 2 = 2(n-3)$. Thus, we can write:
\[
\sum_{i=2}^{n-2} (d_i - 1)(d_i - 2) = \sum_{i=2}^{n-2} d_i^2 - 3 \sum_{i=2}^{n-2} d_i + 2(n-3).
\]

Therefore,
\[
\irr(G) = (d_n - 1)(d_n - 2) + (d_n - 1) + \sum_{i=2}^{n-2} d_i^2 - 3 \sum_{i=2}^{n-2} d_i + 2(n-3) + \sum_{i=1}^{n-1} |d_i - d_{i+1}|.
\]

We suppose that~\eqref{termcatern1} is maximal for $n$ when $d_n > d_1 > \dots > d_2 > d_{n-1}$. 
The question is why we consider that 
\begin{equation}~\label{termcatern1}
\sum_{i=1}^{n-1} |d_i - d_{i+1}| - \sum_{i=2}^{n-1} d_i
\end{equation}
is maximal under this ordering, and we need to prove that this relationship remains maximal for $n+1$.

In this case, from~\eqref{termcatern1} by comparing $\sum_{i=1}^{n-1} |d_i - d_{i+1}| + (d_1 + d_n)$
with the main relationship
\[
(d_n - 1)(d_n - 2) + (d_n - 1) + \sum_{i=2}^{n-2} d_i^2 - 3 \sum_{i=2}^{n-2} d_i + 2(n-3) + \sum_{i=1}^{n-1} |d_i - d_{i+1}|,
\]
we proceed further. Furthermore, consider the sequence
\[
\mathscr{D} = (d_{i1}, d_{i2}, \dots, d_{ik}, d_{j1}, d_{j2}, \dots, d_{j\ell}),
\]
where $k, \ell \in \mathbb{N}$. We observe that if $d_{i1}, d_{i2}, \dots, d_{ik}$ and $d_{j1}, d_{j2}, \dots, d_{j\ell}$ are even, then $|k - \ell| = 0$, and if they are odd, then $|k - \ell| = 1$.

Here, the vertices of the sequence $\mathscr{D}$ occupy different positions, so the Albertson index in this case is defined as the sum of absolute differences over $k-1$ terms:
\[
\irr_{\mathscr{D}}(T) = |d_{i1} - d_{i2}| + \dots + |d_{ik-1} - d_{ik}| + |d_{j1} - d_{j2}| + \dots + |d_{j\ell-1} - d_{j\ell}|.
\]

The sum $|d_{i1} - d_{i2}| + \dots + |d_{ik-1} - d_{ik}|$ telescopes to $|d_{i1} - d_{ik}|$ only if the sequence is strictly monotonic. However, since the values are absolute differences, direct telescoping is generally prohibited. Thus, we have the inequality $|d_{i1} - d_{ik}| \leq |d_{i1} - d_{i2}| + \dots + |d_{ik-1} - d_{ik}|.$

Similarly, for the second sum, $|d_{j1} - d_{j2}| + \dots + |d_{j\ell-1} - d_{j\ell}|$
telescopes to $|d_{j1} - d_{j\ell}|$ if and only if $|d_{j1} - d_{j\ell}| \leq |d_{j1} - d_{j2}| + \dots + |d_{j\ell-1} - d_{j\ell}|.$ Therefore, according to these terms, we have
\[
|d_{i1} - d_{ik}| + |d_{j1} - d_{j\ell}| \leq |d_{i1} - d_{i2}| + \dots + |d_{ik-1} - d_{ik}| + |d_{j1} - d_{j2}| + \dots + |d_{j\ell-1} - d_{j\ell}|.
\]

By considering $d_{i1} = d_{i,1}$ and $d_{i2} = d_{i,2}$, the simplest and most general expression can be conveniently described using summation notation:
\begin{equation}~\label{termcatern2}
\irr_{\mathscr{D}}(T) = \sum_{m=1}^{k-1} |d_{i,m} - d_{i,m+1}| + \sum_{n=1}^{\ell-1} |d_{j,n} - d_{j,n+1}|.
\end{equation}

Assume that $\sum_{i=1}^{n-1} |d_i - d_{i+1}| + d_1 + d_n$
is maximal when $d_n > d_1 > \dots > d_2 > d_{n-1}.$
Therefore,
\[
\sum_{i=1}^{n-1} |d_i - d_{i+1}| + d_1 + d_{n+1}
\]
is maximal when $d_n > d_1 > \dots > d_2 > d_{n-1}.$ We notice that $d_1 + d_{n+1}$ attains its maximum value when $d_{n+1} \geq d_1 \geq \dots \geq d_2 \geq d_n.$

Hence, we need to prove that $\sum_{i=1}^{n-1} |d_i - d_{i+1}|$ is maximal under the same ordering. Thus,
\begin{align*}
\irr(G) &= \sum_{i=1}^{n-1} |d_i - d_{i+1}| + d_1 + d_{n+1} \\
&= |d_1 - d_2| + |d_2 - d_3| + \dots + |d_{n-1} - d_n| + d_1 + d_{n+1}.
\end{align*}

Now, from~\eqref{termcatern1} and \eqref{termcatern2} we consider two cases to determine the maximum value of $\sum_{i=1}^{n-1} |d_i - d_{i+1}|$ according to $\irr(G)$:

\case{1} If $d_i \geq d_{i+1},$ then $\irr(G) = 2 d_1 - d_n + d_{n+1}$ is maximal. Note that $2 d_1$ arises from combining the two $d_1$ terms, and $d_n + d_{n+1}$ is maximal when $d_{n+1} \geq d_1 \geq \dots \geq d_2 \geq d_n.$

\case{2} If $d_i \leq d_{i+1},$ then $\irr(G) = d_n + d_{n+1}$ is maximal, so in this case, the term $\sum_{i=1}^{n-1} |d_i - d_{i+1}|$  is combined with the maximum value $d_1 + d_{n+1}.$

Therefore,
\[
\irr(G) = 
\begin{cases}
2 d_1 - d_n + d_{n+1}, & \text{if } d_i \geq d_{i+1}, \\
d_n + d_{n+1}, & \text{if } d_i \leq d_{i+1},
\end{cases}
\]
is maximal when $d_{n+1} \geq d_1 \geq \dots \geq d_2 \geq d_n,$ 
and $d_1 + d_{n+1}$ is maximal.

This completes the proof.
\end{proof}

In Hypothesize~\ref{hyfibn1}, we provide an application of Albertson index on Fibonacci numbers satisfy caterpillar tree.
\begin{hypothesize}~\label{hyfibn1}
Let $T$ be a tree of order $n$, let $F_n$ be a Fibonacci number with $n>2$. Let $\mathscr{D}=(F_3,\dots,F_n)$ be a degree sequence of Fibonacci number. Then, Albertson index of $T$ among $\mathscr{D}$ is 
\[
\irr(T)=\sum_{i=3}^{n-1}F_i+\sum_{i=5}^{n-1}(F_i-2)\lvert F_i-1\rvert+\lvert F_4-1\rvert+(F_n-1)\lvert F_n-1\rvert+2.
\]
\end{hypothesize}
\begin{proof}
Let $\mathscr{D}=(F_3,\dots,F_n)$ be a degree sequence of Fibonacci number, then we have
\begin{equation}~\label{eqfibn1}
\lvert F_3-F_4\rvert+\dots+\lvert F_{n-1}-F_n\rvert=\sum_{i=3}^{n-1}F_i+2.  
\end{equation}
According to Albertson index for every vertex, by considering the constant term is $\lvert F_4-1\rvert+(F_n-1)\lvert F_n-1\rvert$, we noticed that:
\begin{equation}\label{eqfibn02}
    (F_5-2)\lvert F_5-1\rvert+(F_6-2)\lvert F_6-1\rvert+\dots+(F_{n-1}-2)\lvert F_{n-1}-1\rvert=\sum_{i=5}^{n-1}(F_i-2)\lvert F_i-1\rvert.
\end{equation}
Now, by applying~\eqref{eqfibn1},\eqref{eqfibn02} to Albertson index, we obtain 
\begin{equation}~\label{eqfibn3}
\irr(T)=\sum_{i=3}^{n-1}F_i+\sum_{i=5}^{n-1}(F_i-2)\lvert F_i-1\rvert+\lvert F_4-1\rvert+(F_n-1)\lvert F_n-1\rvert+2.
\end{equation}
As desire.
\end{proof}
\begin{example}
Let $T$ be a tree of order $n$. Let $\mathscr{D}=(F_3,\dots,F_{10})$ be a degree sequence of Fibonacci number, where $F_1=1, F_2=2$ and $F_3=3$ with $F_n=F_{n-1}+F_{n-2}$, then in Figure~\ref{figdegseqfibnumn1} we observed that, then Albertson index according to~\eqref{eqfibn1} is:
\begin{equation}~\label{eqexfibn1}
F_3+\dots+F_{11}=1\sum_{i=3}^{10}F_i+2=143    
\end{equation}
Then, from~\eqref{eqfibn3},\eqref{eqexfibn1} we present Albertson index as
\[
\irr(T)=143+4430+7746=12319.
\]
\begin{figure}[H]
    \centering
    \begin{tikzpicture}[scale=1.1]
\draw   (7,5)-- (7,6);
\draw   (7,6)-- (8,7);
\draw   (8,7)-- (9,8);
\draw   (9,8)-- (10,9);
\draw   (10,9)-- (11,10);
\draw   (11,10)-- (12,9);
\draw   (12,9)-- (13,8);
\draw   (13,8)-- (14,7);
\draw   (14,7)-- (15,6);
\draw   (8,7)-- (8,6);
\draw   (9,8)-- (8.537909109638452,7.0048314346548555);
\draw   (9,8)-- (9,7);
\draw   (9,8)-- (9.564119435380462,7.023835329576004);
\draw   (10,9)-- (9.48810385569587,8.031041760396866);
\draw   (10,9)-- (10.381286916989843,7.974030075633421);
\draw (9.640135015065058,8.2) node[anchor=north west] {$\dots$};
\draw   (12,9)-- (11,8);
\draw   (12,9)-- (12,8);
\draw (11.179450503678076,8.2) node[anchor=north west] {$\dots$};
\draw   (13,8)-- (12,7);
\draw   (13,8)-- (13,7);
\draw (12.224664724341235,7.2) node[anchor=north west] {$\dots$};
\draw   (14,7)-- (13,6);
\draw   (14,7)-- (14,6);
\draw (13.250875050083247,6.1) node[anchor=north west] {$\dots$};
\draw   (15,6)-- (14,5);
\draw   (15,6)-- (15,5);
\draw (14.220073691061812,5.2) node[anchor=north west] {$\dots$};
\draw (6.219433929258355,6.5867457463895915) node[anchor=north west] {$F_3$};
\draw (7.340663729606107,7.726979441658493) node[anchor=north west] {$F_4$};
\draw (8.34787016042697,8.715181977558206) node[anchor=north west] {$F_5$};
\draw (9.222049326799794,9.627368933773328) node[anchor=north west] {$F_6$};
\draw (10.495310286516736,10.729594839199931) node[anchor=north west] {$F_7$};
\draw (12.034625775129753,9.874419567748257) node[anchor=north west] {$F_8$};
\draw (13.155855575477505,8.867213136927393) node[anchor=north west] {$F_9$};
\draw (14.049038636771478,7.917018390869976) node[anchor=north west] {$F_{10}$};
\draw (15.227280121882675,6.757780800679927) node[anchor=north west] {$F_{11}$};
\draw (10.685349235728218,9.532349459167586) node[anchor=north west] {$\dots$};
\begin{scriptsize}
\draw [fill=black] (7,5) circle (1.5pt);
\draw [fill=black] (7,6) circle (1.5pt);
\draw [fill=black] (8,7) circle (1.5pt);
\draw [fill=black] (9,8) circle (1.5pt);
\draw [fill=black] (10,9) circle (1.5pt);
\draw [fill=black] (11,10) circle (1.5pt);
\draw [fill=black] (12,9) circle (1.5pt);
\draw [fill=black] (13,8) circle (1.5pt);
\draw [fill=black] (14,7) circle (1.5pt);
\draw [fill=black] (15,6) circle (1.5pt);
\draw [fill=black] (8,6) circle (1.5pt);
\draw [fill=black] (8.537909109638452,7.0048314346548555) circle (1.5pt);
\draw [fill=black] (9,7) circle (1.5pt);
\draw [fill=black] (9.564119435380462,7.023835329576004) circle (1.5pt);
\draw [fill=black] (9.48810385569587,8.031041760396866) circle (1.5pt);
\draw [fill=black] (10.381286916989843,7.974030075633421) circle (1.5pt);
\draw [fill=black] (11,8) circle (1.5pt);
\draw [fill=black] (12,8) circle (1.5pt);
\draw [fill=black] (12,7) circle (1.5pt);
\draw [fill=black] (13,7) circle (1.5pt);
\draw [fill=black] (13,6) circle (1.5pt);
\draw [fill=black] (14,6) circle (1.5pt);
\draw [fill=black] (14,5) circle (1.5pt);
\draw [fill=black] (15,5) circle (1.5pt);
\end{scriptsize}
\end{tikzpicture}
    \caption{Degree sequence of Fibonacci number}
    \label{figdegseqfibnumn1}
\end{figure}
\end{example}

\section{On Bipartite Graph with Topological Indices}~\label{sec4}
In this section, we delineate certain critical thresholds for both indices: the Sigma Index and the Albertson Index within the bivariate graphical representation. These demarcations hold substantial significance, particularly in the rigorous identification and demarcation of anomalous extremities for each index.

\begin{lemma}~\label{boundbipartitegraphn5}
Let $\mathscr{G}$ be a bipartite graph with $n_1=|V_1(\mathscr{G})|$ and $n_2=|V_2(\mathscr{G})|$, let $\delta\geqslant 2$ be the minimum degree of $\mathscr{G}$. Then, the upper bound of Sigma index satisfy
\begin{equation}~\label{eq1boundbipartitegraphn5}
\sigma(\mathscr{G}) \geqslant \floor*{\dfrac{n_1n_2}{3}}+\Delta(\delta-1)^2+\frac{n_1\Delta+n_2(\delta-1)^2+n_2\Delta+n_1(\delta-1)^3}{n_1n_2+\delta}.
\end{equation}
\end{lemma}
\begin{proof}
Assume $\mathscr{G}$ is a bipartite graph with vertex sets $V_1(\mathscr{G})$ and $V_2(\mathscr{G})$, where $n_1 = |V_1(\mathscr{G})|$ and $n_2 = |V_2(\mathscr{G})|$. Suppose $\delta \geqslant 2$ and consider the condition $\Delta^2 \leq n_1^2 - 2$, where $\delta(n_1^2 - 2) = \delta(n_2^2 - 2)$. Then, the following holds:
\begin{equation}\label{eq2boundbipartitegraphn5}
\sum_{v \in V_1(\mathscr{G})}  d(v) + \sum_{u \in V_2(\mathscr{G})}  d(u) = \delta(n_1^2 - 2).
\end{equation}
From \eqref{eq2boundbipartitegraphn5}, the vertex sets $V_1(\mathscr{G})$ and $V_2(\mathscr{G})$ satisfy
\[
\left\lfloor \frac{n_1 n_2}{3} \right\rfloor \geqslant n_2 (n_1 - \delta) - \delta (n_1 - 2).
\]
When $\mathscr{G}$ is a complete bipartite graph, the Sigma index is given by $\sigma(\mathscr{G}) = n_1 n_2 (n_1 - n_2)^2.$
Thus, we obtain
\begin{equation}\label{eq3boundbipartitegraphn5}
\sigma(\mathscr{G}) > n_2 (n_1 - \delta) - \delta (n_1 - 2) + \left\lfloor \frac{n_1 n_2}{3} \right\rfloor.
\end{equation}
Assuming $\Delta (\delta - 1)^2 < n_2 (n_1 - \delta)$ and $\Delta < \max\{n_1, n_2\}$, \eqref{eq3boundbipartitegraphn5} holds for the term $\Delta (\delta - 1)^2$. This yields
\[
\sigma(\mathscr{G}) > n_2 (n_1 - \delta) + \Delta (\delta - 1)^2 + n_1 (n_2 - \delta).
\]
Considering the term $n_1 n_2 + \delta$, we establish the bound
\begin{equation}\label{eq4boundbipartitegraphn5}
\sigma(\mathscr{G}) \geqslant n_1 n_2 + \delta + \Delta (\delta - 1)^2 + \left\lfloor \frac{n_1 n_2}{3} \right\rfloor.
\end{equation}
From \eqref{eq3boundbipartitegraphn5} and \eqref{eq4boundbipartitegraphn5}, we observe that the term $n_1 \Delta + n_2 (\delta - 1)^2 + n_2 \Delta + n_1 (\delta - 1)^3$
satisfies
\[
n_1 \Delta + n_2 (\delta - 1)^2 \leqslant n_1 n_2 + \delta \quad \text{and} \quad n_2 \Delta + n_1 (\delta - 1)^3 \leqslant n_1 n_2 + \delta,
\]
 by implying that 
\[
n_1 \Delta + n_2 (\delta - 1)^2 + n_2 \Delta + n_1 (\delta - 1)^3 \geqslant n_1 n_2 + \delta.
\]
Moreover, we have
\[
\frac{n_1 \Delta + n_2 (\delta - 1)^2 + n_2 \Delta + n_1 (\delta - 1)^3}{n_1 n_2 + \delta} \leqslant \Delta (n_1 - \delta).
\]
If $\mathscr{G}$ is an incomplete bipartite graph, we establish the bound
\[
\sigma(\mathscr{G}) \leqslant \Delta^3 (\Delta - 1) + \delta (\delta - 1)^3 + \frac{\Delta^3 (\Delta - 1) + 4 n_2^2}{n_1 n_2 + 3 \delta^2}.
\]
Thus, the relations \eqref{eq2boundbipartitegraphn5}--\eqref{eq4boundbipartitegraphn5} remain satisfied for the upper bound of the Sigma index, where
\[
\delta^2 \leqslant \sigma(\mathscr{G}) \leqslant \delta \Delta^2,
\]
and it holds that
\[
\sqrt{2 n_1 \delta^2} \leqslant \sigma(\mathscr{G}) \leqslant \sqrt{2 n_1 n_2 \delta \Delta^2}.
\]
Taking into account \eqref{eq2boundbipartitegraphn5}, the required relation \eqref{eq1boundbipartitegraphn5} is satisfied in both the complete and incomplete bipartite graph.

This completes the proof.

\end{proof}

In fact, Theorem~\ref{sharpbound} provide us an important result of the bound with $\sigma_{\max}$ on $\mathscr{G}$. 
\begin{theorem}~\label{sharpbound}
Let $\mathscr{G}$ be a bipartite graph, with $n_1 = |V_1(\mathscr{G})|$ and $n_2 = |V_2(\mathscr{G})|$. The maximum value of the Sigma index is bounded by $n_1^4 + n_2^4$ if and only if $\mathscr{G}$ has a sharp bound.
\end{theorem}

Assume $\theta(n^2)=n_1^4 + n_2^4$ and $\theta(n)=n_1^2 + n_2^2$, then, by Theorem~\ref{sharpboundn2} we presented the upper bound and the lower bound of $\sigma$. 

\begin{theorem}~\label{sharpboundn2}
If $\mathscr{G}$ has a sharp bound $\sigma_{\max}(\mathscr{G})$, then when $\irr(\mathscr{G}) \leqslant \sigma(\mathscr{G}) \leqslant \theta(n^2)$, the upper bound of the Sigma index satisfies 
\begin{equation} \label{eq1sharpboundn2}
\sigma(\mathscr{G}) \geqslant \irr(\mathscr{G}) + \theta(n).
\end{equation}
Similarly, the lower bound of the Sigma index satisfies
\begin{equation} \label{eq2sharpboundn2}
\sigma(\mathscr{G}) \leqslant \irr(\mathscr{G}) + \theta(n^2).
\end{equation}
\end{theorem}
\begin{proof}
Suppose $\theta(n^2) = n_1^4 + n_2^4$ and $\theta(n) = n_1^2 + n_2^2$, where the Albertson index on $\mathscr{G}$ satisfies  
\[
\sqrt{\sigma(\mathscr{G})} \leq \irr(\mathscr{G}) \leq \sqrt{(n_1+n_2) \sigma(\mathscr{G}) + n_1^2 + n_2^2}.
\]  
According to Theorem~\ref{sharpbound}, we find that the maximum value of the Sigma index is bounded by $\theta(n^2)$, and this holds for the Sigma index bounded by $\theta(n^2)$.

\case{1} Assume $\sigma(\mathscr{G}) \geqslant \irr(\mathscr{G}) + \theta(n)$ is not the upper bound of the Sigma index. Then, we notice that  
\[
\sqrt{\sigma_{\min}(\mathscr{G})} \leq \irr(\mathscr{G}) \leq \sqrt{(n_1 + n_2)\sigma_{\max}(\mathscr{G})}.
\]  
Therefore, we have  
\begin{equation} \label{eq3sharpboundn2}
\sigma_{\max}(\mathscr{G}) \geqslant \theta(n^2), \quad \sigma_{\min}(\mathscr{G}) \geqslant \theta(n).
\end{equation}  
Thus, we note that $\sigma_{\min}(\mathscr{G}) \leqslant \theta(n^2)$. Additionally, for $\delta \geqslant 3$, the Sigma index satisfies the bound  $\sigma(\mathscr{G}) > n_2 (n_1 - \delta) + n_1 (n_2 - \delta).$ 
Then,  
\[
\sigma(\mathscr{G}) - \irr(\mathscr{G}) > n_1^4 \quad \text{and} \quad \sigma(\mathscr{G}) - \irr(\mathscr{G}) > n_2^4.
\]  
Thus, it should be $\sigma(\mathscr{G}) - \irr(\mathscr{G}) \leqslant \theta(n^2).$ 
That means $\sigma(\mathscr{G}) \geqslant \irr(\mathscr{G}) + \theta(n)$ is the upper bound of the Sigma index.

\case{2} In this case, we need to prove that~\eqref{eq2sharpboundn2} is the lower bound of the Sigma index given the term.  
If $\mathscr{G}$ is a complete bipartite graph, it satisfies  $\irr(\mathscr{G}) = n_1^2 n_2 - n_2^2 n_1.$ 
Then, according to~\eqref{eq3sharpboundn2}, we find that  $\sigma(\mathscr{G}) \geqslant \theta(n^2).$ 
Thus, by considering  $\sigma(\mathscr{G}) = n_1 n_2 (n_1 - n_2)^2,$ 
we have the bound~\eqref{eq4sharpboundn2}, which presents the relationship with the Albertson index as  
\begin{equation} \label{eq4sharpboundn2}
\sigma_{\min}(\mathscr{G}) \leqslant (n_1 + n_2) \irr(\mathscr{G}) \leqslant \sigma_{\max}(\mathscr{G}) + \theta(n^2).
\end{equation}  
Now, according to Case 1, we find that~\eqref{eq1sharpboundn2} is the upper bound of the Sigma index satisfying  $\sigma(\mathscr{G}) \geqslant \irr(\mathscr{G}) + \theta(n).$
Then, when $\sigma(\mathscr{G}) \geqslant \theta(n^2)$, it holds that  
\[
\sigma(\mathscr{G}) \leqslant (n_1 + n_2) \irr(\mathscr{G}) + \theta(n^2).
\]  
Thus, we have  
\begin{equation} \label{eq5sharpboundn2}
\sigma(\mathscr{G}) \leqslant (n_1 + n_2) \irr(\mathscr{G}) + \theta(n).
\end{equation}  
Actually, both~\eqref{eq4sharpboundn2} and \eqref{eq5sharpboundn2} remain correct when $\mathscr{G}$ is an incomplete bipartite graph. In this case, the bound holds as  
\begin{equation} \label{eq6sharpboundn2}
\sigma(\mathscr{G}) \leqslant (n_1 |n_1 - n_2| + n_2 |n_1 - n_2|) \irr(\mathscr{G}) + \theta(n).
\end{equation}  
Therefore, from~\eqref{eq4sharpboundn2},~\eqref{eq5sharpboundn2}, and~\eqref{eq6sharpboundn2}, it follows that~\eqref{eq2sharpboundn2} is valid.

\end{proof}

\section{Conclusion}\label{sec5}
Through this paper, the study of the extreme values (maximum and minimum) for each index requires certain conditions. According to Proposition~\ref{boundalbertsonn02},  the upper bound (the lower bound) of the minimum Albertson index is: 
    \begin{equation*}
\irr_{\min}\geqslant \frac{(\Delta-2)^3}{n\Delta-\delta}, \quad \irr_{\min}\leqslant \frac{\delta}{\Delta+1}n\Delta^2+\frac{\Delta^2(\Delta-\delta)}{6\delta(\Delta-1)}.
\end{equation*}
Through Proposition~\ref{proboundmindel} we established the maximum value with the minimum degree $\delta$ with condition $2(\delta-1)(2\delta-1)>0$. Lemma~\ref{lemmaboundAlbertsonn2} focuses on a single tree with general degree properties, introducing a parameterized lower bound that depends on the maximum and minimum degrees, modulated by a scaling factor $\alpha$, where the lower bound of Albertson index is 
 \[\irr(T) \leq \floor*{\frac{3n^2-10n}{2}}2^{\alpha}+\Delta^2-n\delta.\] 
The investigation of topological indices across arboreal structures, through the discernment of aberrant values, constitutes the theoretical underpinning for myriad pragmatic applications, a salient exemplar of which we have delineated in the present study.

\section*{Acknowledgements}
The authors would like to express their sincere gratitude to the anonymous reviewers for their insightful comments and constructive suggestions, which have substantially enhanced the clarity and quality of this manuscript.

\section*{Declarations}
\begin{itemize}
	\item Funding: Not Funding.
	\item Conflict of interest/Competing interests: The author declare that there are no conflicts of interest or competing interests related to this study.
	\item Ethics approval and consent to participate: The author contributed equally to this work.
	\item Data availability statement: All data is included within the manuscript.
\end{itemize}

\end{document}